\documentclass[12pt]{article}
\usepackage{txfonts}
\usepackage{graphicx}
\usepackage{tikz}

\def\T{\Gamma} \def\D{\Delta} \def\Th{\Theta}
\def\Ld{\Lambda} \def\E{\Sigma} \def\O{\Omega}
\def\a{\alpha} \def\b{\beta} \def\g{\gamma} \def\d{\delta} \def\e{\varepsilon}
\def\r{\rho} \def\o{\sigma} \def\t{\tau} \def\w{\omega} \def\k{\kappa}
\def\th{\theta} \def\ld{\lambda} \def\ph{\varphi} \def\z{\zeta}
\def\A{$A$~} \def\G{$G$~} \def\H{$H$~} \def\K{$K$~} \def\M{$M$~} \def\N{$N$~}
\def\P{$P$~} \def\Q{$Q$~} \def\R{$R$~} \def\V{$V$~} \def\X{$X$~} \def\Y{$Y$~}
\def\rmA{{\bf A}} \def\rmD{{\bf D}} \def\rmS{{\bf S}} \def\rmK{{\bf K}}
\def\rmM{{\bf M}} \def\Z{{\Bbb Z}} \def\GL{{\bf GL}} \def\C{Cayley}
\def\oa{\ovl A} \def\og{\ovl G} \def\oh{\ovl H} \def\ob{\ovl B} \def\oq{\ovl Q}
\def\oc{\ovl C} \def\ok{\ovl K} \def\ol{\ovl L} \def\om{\ovl M} \def\on{\ovl N}
\def\op{\ovl P} \def\oR{\ovl R} \def\os{\ovl S} \def\ot{\ovl T} \def\ou{\ovl U}
\def\ov{\ovl V} \def\ow{\ovl W} \def\ox{\ovl X} \def\oT{\ovl\T}
\def\lg{\langle} \def\rg{\rangle}
\def\di{\bigm|} \def\Di{\Bigm|} \def\nd{\mathrel{\bigm|\kern-.7em/}}
\def\Nd{\mathrel{\not\,\Bigm|}} \def\edi{\bigm|\bigm|}
\def\m{\medskip} \def\l{\noindent} \def\x{$\!\,$}  \def\J{$-\!\,$}
\def\Hom{\hbox{\rm Hom}} \def\Aut{\hbox{\rm Aut}} \def\Inn{\hbox{\rm Inn}}
\def\Syl{\hbox{\rm Syl}} \def\Sym{\hbox{\rm Sym}} \def\Alt{\hbox{\rm Alt}}
\def\Ker{\hbox{\rm Ker}} \def\fix{\hbox{\rm fix}} \def\mod{\hbox{\rm mod}}
\def\psl{{\bf P\!SL}} \def\Cay{\hbox{\rm Cay}} \def\Mult{\hbox{\rm Mult}}
\def\val{\hbox{\rm Val}} \def\Sab{\hbox{\rm Sab}} \def\supp{\hbox{\rm supp}}
\def\qed{\hfill $\Box$} \def\qqed{\qed\vspace{3truemm}}
\def\CS{\Cay(G,S)} \def\CT{\Cay(G,T)}
\def\h{\heiti\bf} \def\hs{\ziti{E}\bf} \def\st{\songti} \def\ft{\fangsong}
\def\kt{\kaishu} \def\heit{\hs\relax} \def\songt{\st\rm\relax}
\def\fangs{\ft\rm\relax} \def\kaish{\kt\rm\relax} \def\fs{\footnotesize}
\begin{document}

\newtheorem{theorem}{Theorem}[section]
\newtheorem{corollary}[theorem]{Corollary}
\newtheorem{definition}[theorem]{Definition}
\newtheorem{conjecture}[theorem]{Conjecture}
\newtheorem{question}[theorem]{Question}
\newtheorem{lemma}[theorem]{Lemma}
\newtheorem{proposition}[theorem]{Proposition}
\newtheorem{example}[theorem]{Example}
\newtheorem{problem}[theorem]{Problem}
\newenvironment{proof}{\noindent {\bf
Proof.}}{\rule{3mm}{3mm}\par\medskip}
\newcommand{\remark}{\medskip\par\noindent {\bf Remark.~~}}
\newcommand{\pp}{{\it p.}}
\newcommand{\de}{\em}

\def\dfrac{\displaystyle\frac} \def\ovl{\overline}
\def\for{\forall~} \def\exi{\exists~} \def\c{\subseteq}
\def\iif{\Longleftrightarrow} \def\Rto{\Rightarrow} \def\Lto{\Leftarrow}
\def\T{\Gamma} \def\D{\Delta} \def\Th{\Theta}
\def\Ld{\Lambda} \def\E{\Sigma} \def\O{\Omega}
\def\a{\alpha} \def\b{\beta} \def\g{\gamma} \def\d{\delta} \def\e{\varepsilon}
\def\r{\rho} \def\o{\sigma} \def\t{\tau} \def\w{\omega} \def\k{\kappa}
\def\th{\theta} \def\ld{\lambda} \def\ph{\varphi} \def\z{\zeta}
\def\A{$A$~} \def\G{$G$~} \def\H{$H$~} \def\K{$K$~} \def\M{$M$~} \def\N{$N$~}
\def\P{$P$~} \def\Q{$Q$~} \def\R{$R$~} \def\V{$V$~} \def\X{$X$~} \def\Y{$Y$~}
\def\rmA{{\bf A}} \def\rmD{{\bf D}} \def\rmS{{\bf S}} \def\rmK{{\bf K}}
\def\rmM{{\bf M}} \def\Z{{\Bbb Z}} \def\GL{{\bf GL}} \def\C{Cayley}
\def\oa{\ovl A} \def\og{\ovl G} \def\oh{\ovl H} \def\ob{\ovl B} \def\oq{\ovl Q}
\def\oc{\ovl C} \def\ok{\ovl K} \def\ol{\ovl L} \def\om{\ovl M} \def\on{\ovl N}
\def\op{\ovl P} \def\oR{\ovl R} \def\os{\ovl S} \def\ot{\ovl T} \def\ou{\ovl U}
\def\ov{\ovl V} \def\ow{\ovl W} \def\ox{\ovl X} \def\oT{\ovl\T}
\def\lg{\langle} \def\rg{\rangle}
\def\di{\bigm|} \def\Di{\Bigm|} \def\nd{\mathrel{\bigm|\kern-.7em/}}
\def\Nd{\mathrel{\not\,\Bigm|}} \def\edi{\bigm|\bigm|}
\def\m{\medskip} \def\l{\noindent} \def\x{$\!\,$}  \def\J{$-\!\,$}
\def\Hom{\hbox{\rm Hom}} \def\Aut{\hbox{\rm Aut}} \def\Inn{\hbox{\rm Inn}}
\def\Syl{\hbox{\rm Syl}} \def\Sym{\hbox{\rm Sym}} \def\Alt{\hbox{\rm Alt}}
\def\Ker{\hbox{\rm Ker}} \def\fix{\hbox{\rm fix}} \def\mod{\hbox{\rm mod}}
\def\psl{{\bf P\!SL}} \def\Cay{\hbox{\rm Cay}} \def\Mult{\hbox{\rm Mult}}
\def\val{\hbox{\rm Val}} \def\Sab{\hbox{\rm Sab}} \def\supp{\hbox{\rm supp}}
\def\qed{\hfill $\Box$} \def\qqed{\qed\vspace{3truemm}}
\def\CS{\Cay(G,S)} \def\CT{\Cay(G,T)}
\def\h{\heiti\bf} \def\hs{\ziti{E}\bf} \def\st{\songti} \def\ft{\fangsong}
\def\kt{\kaishu} \def\heit{\hs\relax} \def\songt{\st\rm\relax}
\def\fangs{\ft\rm\relax} \def\kaish{\kt\rm\relax} \def\fs{\footnotesize}

\newcommand{\JEC}{{\it Europ. J. Combinatorics},  }
\newcommand{\JCTB}{{\it J. Combin. Theory Ser. B.}, }
\newcommand{\JCT}{{\it J. Combin. Theory}, }
\newcommand{\JGT}{{\it J. Graph Theory}, }
\newcommand{\ComHung}{{\it Combinatorica}, }
\newcommand{\DM}{{\it Discrete Math.}, }
\newcommand{\ARS}{{\it Ars Combin.}, }
\newcommand{\SIAMDM}{{\it SIAM J. Discrete Math.}, }
\newcommand{\SIAMADM}{{\it SIAM J. Algebraic Discrete Methods}, }
\newcommand{\SIAMC}{{\it SIAM J. Comput.}, }
\newcommand{\ConAMS}{{\it Contemp. Math. AMS}, }
\newcommand{\TransAMS}{{\it Trans. Amer. Math. Soc.}, }
\newcommand{\AnDM}{{\it Ann. Discrete Math.}, }
\newcommand{\NBS}{{\it J. Res. Nat. Bur. Standards} {\rm B}, }
\newcommand{\ConNum}{{\it Congr. Numer.}, }
\newcommand{\CJM}{{\it Canad. J. Math.}, }
\newcommand{\JLMS}{{\it J. London Math. Soc.}, }
\newcommand{\PLMS}{{\it Proc. London Math. Soc.}, }
\newcommand{\PAMS}{{\it Proc. Amer. Math. Soc.}, }
\newcommand{\JCMCC}{{\it J. Combin. Math. Combin. Comput.}, }
\newcommand{\GC}{{\it Graphs Combin.}, }

\title{  Trees with given degree sequences that have
minimal subtrees \thanks{ This work is supported by the National
Natural Science Foundation of China (No:10971137)
\newline \indent
 $^{\dagger}$Corresponding author: Xiao-Dong
Zhang (Email: xiaodong@sjtu.edu.cn)}}
\author{Xiu-Mei Zhang$^{1,2}$, Xiao-Dong Zhang$^{1,\dagger}$\\
{\small $^{1}$Department of Mathematics,
 Shanghai Jiao Tong University} \\
{\small  800 Dongchuan road, Shanghai, 200240, P. R. China}\\
{\small $^{2}$Department of Mathematics,Shanghai sandau University}\\
{\small  2727 jinhai road, Shanghai, 201209, P. R. China}\\
 }
\date{}
\maketitle
 \begin{abstract}

In this paper, we investigate the structures of an extremal tree
which has the minimal number of subtrees in the set of all trees
with the given  degree sequence of a tree. In particular, the
extremal trees must be caterpillar and but in general not unique.
Moreover, all extremal trees
   with a given degree sequence $\pi=(d_1, \cdots, d_5, 1,\cdots, 1)$ have
   been characterized.
 \end{abstract}


 \vskip 0.5cm

\section{Introduction}
Let  $T=(V,E)$ be a tree with vertex set $V(T)$ and edge set $E(T)$.
 vertices of degree 1 of $T$ are called {\it leaves}.
 For any two vertices $u, v\in V(T)$, the distance between two vertices $u$ and $
 v$, denoted by $d_T(u,v)$ (or $d(u,v)$
 for short), is length of the unique path $P_{T}(u,v)$ joining $u$ and $v$ in $T$.
   Then
$D(T)=max\{d(u,v)|u,v \in V(T)\}$ is the  diameter of tree $T$.
 Moreover, we
use $N_T(v)$ to indicate the neighbors of vertex $v$ and
$d(v)=|N_T(v)|$ is the degree of $v$. A {\it caterpillar} is a tree,
which has a path, such that every vertex not on the path is adjacent
to some vertex on the path.

For a tree $T=(V(T), E(T))$ and
$v_{1}$,$v_{2}$,\dots,$v_{m-1}$,$v_{m}$ $\in V(T)$, let
$f_{T}(v_{1},v_{2},\dots,v_{m-1},v_{m}$) denote the number of
subtrees of $T$ that contain the vertices $v_{1}, v_{2}, \dots,
v_{m-1}, v_{m}$. In particular, $f_{T}(v)$ denotes the number of
subtrees of $T$ that contain $v$. Let $\varphi(T)$ denote the number
of non-empty subtrees of $T$. For other terminology and notions, we
can follow from \cite{bondy1976}.

The number of subtrees of a tree has received much attention, since
 it can reveal some different
structures and  characterization of a tree. It is well known that
the path and the star $K_{1,n-1}$ have the smallest  and largest
numbers of  subtrees among all trees of order $n$, respectively.
Roughly speaking, the less branched the tree is, the smaller the
number of subtrees.  A observation is that trees with the same
maximum degree appear to be clustered together in this order. One
may wonder which trees have the largest or smallest Wiener index
under the restriction of maximum degree. Kirk and Wang
\cite{kirk2008} characterized the extremal tree with given a order
and maximum vertex degree which have the largest number of subtrees.
Recently, Zhang et.al.  determined the extremal trees with give a
degree sequence that have the largest number of subtrees.  For other
related results, the authors can be referred to \cite{Szelely2005,
Szelely2007}. On the other hand,  Heuberger and Prodinger
\cite{Heuberger2007} presented formulas  to calculate the number of
subtrees of extremal trees among binary trees. Yen and Yeh
\cite{yan2006} gave a linear-time algorithm to count the subtrees of
a tree.  Eistenstat and Gordon \cite{Eisenstat1995} constructed two
non-isomorphic trees that have the same the number of subtrees,
which are related to the greedoid Tutte polynomial of a tree.
Further an interesting fact is that among  above every kind of
trees, the extremal one that maximizes the number of subtrees is
exactly the one that minimizes some chemical indices such as the
well known $Wiener$ $index$ (see \cite{wiener1947} for details) and
vice versa.

Although a counter example has showed in \cite{wagner2007}
that no 'nice' functional relationship exits between these two concepts,
the results are extended to some other kind trees.
Such as, recently, the extremal one that maximizes the number of subtrees
among trees with a given degree sequence are characterized
 in \cite{zhang2011} and the extremal structures
 once again coincide with the once found for the Wiener
 index \cite{Hua2007} and \cite{X.D.Zhang2008}, respectively.
  When the extremal one that maximizes the Wiener index with a given degree sequence
  are istudied in \cite{shi1993} and \cite{zhang2010}.
   Then it is natural to consider the following
question.


\begin{problem}
Given the degree sequence and the number of vertices of a tree, find the lower
bound for the number of subtrees, and characterize all
extremal trees that attain this bound.
\end{problem}

It will not be a surprise to see that such extremal trees coincide
with the ones that attain the maximal Wiener index. The rest of the
paper is organized as follows. In Section 2, we prove that a minimum
optimal tree must be a caterpillar. In Section 3, we discuss some
properties of the extremal tree with minimal (maximal) number of
subtrees among caterpillar trees with given order and degree
sequence. In Section 4, the extremal trees with minimal subtrees
among given degree sequence $\pi=(d_1,d_2, \cdots, d_n)$, where $d_1
\ge \cdots d_k \ge 2>d_{k+1}=1$ and $k\le 5$ are characterized.
Moreover, the extremal minimal trees are not unique.

\section{Properties of optimal minimal trees with a given degree sequence }

For a nonincreasing sequence of positive integers $\pi=(d_1, d_2,\cdots, d_k,\cdots d_{n})$,
 $d_1\ge d_2\ge \cdots \ge d_k\ge 2$ and $d_{k+1}=\cdots =d_n=1$.
 If $\pi$ is the degree sequence of a tree, let ${\mathcal{T}}_{\pi}$ denote the set of all trees with $\pi$ as its degree sequence.
For convenience, we refer to trees that maximize (minimize) the number of subtrees as
 maximum (minimum) optimal. The main result of this section can be
 stated as follow.

\begin{theorem}\label{theorem2-1}  Let
  $\pi=(d_1, d_2,\cdots,
d_k,\cdots
 d_{n})$ be the degree sequence of a tree
 with  $d_1\ge d_2\ge \cdots \ge d_k\ge 2$ and $d_{k+1}=\cdots =d_n=1$.
If $T^{\ast} $ is  a minimum optimal tree in  $
{\mathcal{T}}_{\pi},$ then $T^{\ast}$ must be a caterpillar.
\end{theorem}
\begin{proof} Let $T^{\ast} \in {\mathcal{T}_{\pi}}$ be a minimum optimal tree.
If the diameter $D(T^{\ast})$ is equal to 2, then $T^{\ast}$ is
$K_{1,n-1}$, and also is caterpillar. If $D(T^{\ast})=3$, then the
degree sequence of $T^{\ast}$ must be $\pi=(d_1,d_2, 1,\cdots, 1)$
and $ d_1\ge d_2\ge 2$. It is easy to see that  $T^{\ast}$ is a
caterpillar. Hence we only need to prove the assertion for
$D(T^{\ast})\ge 4$ and at least three internal vertices. Let
 $P=v_0v_1v_2\cdots v_r$ be the longest  path in $T^{\ast}$, where
 $r=D(T^{\ast})$. Then
  $d(v_0)=d(v_{r})=1$  and $d(v_1), \cdots, d(v_{r-1})\ge 2$,
  which implies that there are at
  least $r-1$ vertices with at least degree 2. So $k\ge r-1$. Now we
  have the following claim

  {\bf Claim: }  $k=r-1$.

If $k>r-1$ then $T^{\ast}$ is not a caterpillar. Thus there exists a
vertex $y\not\in V(P)$ and  $2\le l\le r-2$ such that the edge
$yv_l\in E(T)$ and $N_T(y)=\{v_l,x_1,x_2,\cdots,x_s\}, s\ge1$.
Moreover, $T-\{v_lv_{l+1}, v_ly\}$ has  three connected components,
$W_1, W_2, W_3$ which contain $v_l, v_{l+1}, y$, respectively.
Without loss of generality, we assume that
$f_{T_1}(v_l)>f_{T_2}(v_{l+1})$. Further  let $V_i$ be the connected
component of $T-\{v_{i-1}v_{i}, v_iv_{i+1}\}$ containing vertex
$v_i$  and $a_i=f_{V_i}(v_i)$  for $l\le i\le r-1$ (for convenience,
$a_r=1$). Moreover, denote $b_l=f_{W_1}(v_l)$ and $a=f_{W_3}(y)>1$.
Then
\begin{equation}\label{equation2-1}
b_l> a_{l+1}(1+a_{l+2}+a_{l+2}a_{l+3}+\cdots+a_{l+2}a_{l+3}\cdots
a_{r}).
\end{equation}
Let $T^{\prime}$ be a tree with degree sequence $\pi$ obtained  from
$T^{\ast}$ by deleting the edges $yx_1,yx_2,\cdots,$ $yx_s$ in
$T^{\ast}$ and adding the edges $v_rx_1,v_rx_2,\cdots,v_rx_s$.
Obviously, $T^{\prime}\in {\mathcal{T}}_{\pi}$.
 Let $W_3^{\prime}$ be the
connected component of $T^{\prime}-\{v_{r-1}v_{r}\}$ containing vertex $v_r$.
Then $W_3^{\prime}$ is isomorphic to $W_3$ and
$f_{W_3^{\prime}}(v_r)=f_{W_3}(y)=a$. Clearly, the number of
subtrees of $T^{\ast}$ with containing $v_r, y$  is equal to the
number of
 subtrees of $T^{\prime}$ with containing $v_r, y$. The number of
  of subtrees of $T^{\ast}$ without containing $v_r, y$ is equal
  to the number of
 subtrees of $T^{\prime}$ without containing $v_r, y$.
 The number of
 of subtrees of $T^{\ast}$ with containing $y$ and no containing
 $v_r$ is equal to $a(1+b_l+b_la_{l+1}+\cdots +b_la_{l+1}\cdots
a_{r-1})$, while  the number of
 of subtrees of $T^{\prime}$ with containing $y$ and no containing
 $v_r$ is equal to $1+b_l+b_la_{l+1}+\cdots
+b_la_{l+1}\cdots a_{r-1}$. The number of
 of subtrees of $T^{\ast}$ with containing $v_r$ and no containing
 $y$ is equal to  $1+a_{r-1}+a_{r-1}a_{r-2}+\cdots+a_{r-1}a_{r-2}\cdots
a_{l+1}b_l$ , while  the number of
 of subtrees of $T^{\prime}$ with containing $v_r$ and no containing
 $y$ is equal to $a(1+a_{r-1}+a_{r-1}a_{r-2}+\cdots+a_{r-1}a_{r-2}\cdots
a_{l+1}b_l)$. Hence by equation (\ref{equation2-1}) and $a>1$,
\begin{eqnarray*}
\varphi(T^{\prime})-\varphi(T^{\ast})&=&
(1-a)[b_l+b_la_{l+1}+b_la_{l+1}a_{l+2}+\cdots
+b_la_{l+1}a_{l+2}\cdots a_{r-2}\\
&&-a_{r-1}-a_{r-1}a_{r-2}-\cdots -a_{r-1}a_{r-2}\cdots\ a_{l+1}]\\
&&<0.
\end{eqnarray*}
It contradicts to $T^{\ast}$ being an minimum optimal tree. Hence
the claim holds and  $T^{\ast}$ is a caterpillar.
\end{proof}

\section{Properties of optimal trees
among caterpillars with a given degree sequence}

In this section, we study some properties of optimal minimal
(maximal) trees in the set of all caterpillars for a given degree
sequence, since an optimal minimal tree in the set of all trees with
a given degree sequence must be caterpillar. For graphic sequence
$\pi=(d_1, d_2,\cdots, d_k,\cdots d_{n})$ of a tree with $d_1\ge
d_2\ge \cdots \ge d_k\ge 2$ and $d_{k+1}=\cdots =d_n=1$ with $k\ge
2$, let
$$\mathcal{C}_{\pi}=\{T: \ \ T \ {\rm is\  a\  caterpillar
\ with\ degree\ sequence }\ \pi\}.$$
 If $(y_1,y_2,\cdots,y_k)$ is a permutation of $(d_1-2, d_2-2,\cdots, d_k-2)$,
 then the caterpillar $C(y_1, \cdots, y_k)$ is obtained from a path $v_0v_1v_2\cdots v_kv_{k+1}$
by adding $y_1, \cdots, y_{k}$ pendent edges at $v_1, \cdots, v_k$,
respectively. Clearly, $C(y_1, \cdots, y_k)\in \mathcal{C}_{\pi}$.
Conversely, for any $T\in\mathcal{C}_{\pi}$, $T$ can be obtained in
this way. Moreover, let $V_j$, $V_{\ge j}$ and $V_{\le j}$ denote
the connected component of $C(y_1,y_2,\cdots,y_k)$ containing $v_j$
after deleting the two edges $v_{j-1}v_j$ and $v_{j}v_{j+1}$,
 the edge $v_{j-1}v_j  $, and the edge $v_{j}v_{j+1}$, respectively,  for
$ j=1, \cdots, k$. For convenience, let $V_0=V_{\le 0}=\{v_0\}$ and
$V_{k+1}=V_{\ge k+1}=\{v_{k+1}\}$.
\begin{lemma}\label{lemma3-2}
  Let $T$ be a tree $C(y_1, \cdots, y_k)$ in $\mathcal{C}_{\pi}$ with
the spine $v_0v_1\cdots v_{k+1}$.
   If there exists a $2 \le p \le k-1$  such that $f_{V_{p-i}}(v_{p-i})\ge
  f_{V_{p+i}}(v_{p+i})$ for $i=1, \cdots, q$ and $q \le min\{k-p,p-1\}$  with at least one strict
  inequality  and $f_{V_{\le p-q-1}}(v_{p-q-1})>
  f_{V_{\ge p+q+1}}(v_{p+q+1})$, then there exists
a caterpillar $T_1\in \mathcal{C}_{\pi}$ such that
$$\varphi(T_1)<\varphi(T).$$
  \end{lemma}
\begin{proof}   Let $W$ be the connected
component of $T$ by deleting the two edges $v_{p-q-1}v_{p-q}$ and
$v_{p+q}v_{p+q+1}$ and containing vertices $v_{p-q}$ and $ v_{p+q}$.
Let $X$ be obtained from the $V_{\le p-q}$ by adding the edge
$v_{p-q-1}v_{p-q}$ and let
 $Y$  be obtained from the $V_{\ge p+q+1}$ by adding the edge
 $v_{p+q}v_{p+q+1}$. Then $f_X(v_{p-q})>f_Y(v_{p+q})$. Further,
 by Lemma 3.1 in \cite{kirk2008}, $f_W(v_{p-q})>f_W(v_{p+q})$.
 Now let $T_1$ be the caterpillar from $T$ by deleting two edges
 $v_{p-q-1}v_{p-q}$ and $v_{p+q}v_{p+q+1}$ and adding two edges
 $v_{p-q-1}v_{p+q}$ and $v_{p+q}v_{p-q-1}$. Then $T_1\in
 \mathcal{C}_{\pi}$. By Lemma 3.2 in \cite{kirk2008},
 $\varphi(T_1)<\varphi(T)$. Hence the assertion holds.
\end{proof}

Similarly, we can prove the following assertion by the same method
and omit the detail.

\begin{lemma}\label{lemma3-2-1}
 Let $T$ be a tree $C(y_1, \cdots, y_k)$ in $\mathcal{C}_{\pi}$ with
the spine $v_0v_1\cdots v_{k+1}$.
   If there exists a $1 \le p \le k-1$  such that $f_{V_{p-i}}(v_{p-i})\ge
  f_{V_{p+i+1}}(v_{p+i+1})$ for $i=0, 1, \cdots, q$ and
  $q \le min\{k-p-1, p-1\}$  with at least one strict
  inequality  and $f_{V_{\le p-q-1}}(v_{p-q-1})>
  f_{V_{\ge p+q+2}}(v_{p+q+2})$, then there exists
a caterpillar $T_1\in \mathcal{C}_{\pi}$ such that
$$\varphi(T_1)<\varphi(T).$$
  \end{lemma}
It follows from Lemmas~\ref{lemma3-2} and \ref{lemma3-2-1} that we
have got a property of an optimally minimal caterpillar tree in
$\mathcal{C}_{\pi}$.

\begin{corollary}\label{cor3-3}
Let $T$ be a minimum optional tree  $C(z_1, \cdots, z_k)$ with the
spine
  $v_0v_1v_2\cdots v_kv_{k+1}$ in $\mathcal{C}_{\pi}$.

  (i).    If  there exists a $2 \le p \le k-1$ such that
  $f_{V_{p-i}}(v_{p-i})\ge \ \ ({\rm or}\  \le)\ \
  f_{V_{p+i}}(v_{p+i})$ for $i=1, \cdots, q$ and $q \le min\{k-p,p-1\}$ with at least one strict
  inequality, then
  \begin{equation}\label{cor3-1-1}f_{V_{\le p-q-1}}(v_{p-q-1})\le \ \ ({\rm or}\
  \ge)\ \
     f_{V_{\ge p+q+1}}(v_{p+q+1}).
     \end{equation}

  (ii). If there exists a  $1 \le p \le k-1$ such that
  $f_{V_{p-i}}(v_{p-i})\ge\ \ ({\rm or}\ \le)\ \
  f_{V_{p+i+1}}(v_{p+i+1})$  $i=0, \cdots, q$ and
  $q \le min\{k-p-1,p-1\}$ with at least one strict
  inequality, then
  \begin{equation}\label{cor3-1-2}f_{V_{\le p-q-1}}(v_{p-q-1})\le \ \ ({\rm or}\ \ge)\ \
  f_{V_{\ge p+q+2}}(v_{p+q+2}).
  \end{equation}
\end{corollary}

Further, we present another property of a minimal optimal tree in
$\mathcal{C}_{\pi}$.

\begin{theorem}\label{theorem3-5} Let $\pi=(d_1, d_2,\cdots, d_k,\cdots d_{n})$
be the degree sequence of a tree with $d_1\ge d_2\ge \cdots \ge
d_k\ge 2$ and $d_{k+1}=\cdots =d_n=1$, where $k\ge 3$. If
$C(z_1,z_2,\cdots,z_k)$ is a minimal optional caterpillar in
$\mathcal{C}_{\pi}$ with $z_1\ge z_k$, then there exists a positive
integer number $1\le t\le k-1$ such that
\begin{eqnarray*}
z_1\ge z_2\cdots \ge z_{t-1}>z_t=d_k-2
\end{eqnarray*}
and
\begin{eqnarray*}
z_t\le z_{t+1}\cdots \le z_k.
\end{eqnarray*}
\end{theorem}

\begin{proof}.
We consider the following three cases.

{\bf Case 1:} $d_1=\cdots=d_k$. Clearly, the assertion holds.

{\bf Case 2:} $d_1>d_2=\cdots=d_k$. Suppose (for contradiction) that
there exists $2\le l\le k-1$ such that $z_l=d_1-2$. If $l$ is odd,
let $2\le p=\frac{l+1}{2}\le k-1$ and $q=\frac{l-1}{2}$. Then
$f_{V_{p-i}}(v_{p-i})=f_{V_{p+i}}(v_{p+i})$ for $i=1, \cdots, q-1$
and $f_{V_{p-q}}(v_{p-q})=2^{d_2-2}<2^{d_1-2}=f_{V_{p+q}}(v_{p+1})$.
Hence by (i) of  Corollary~\ref{cor3-3}, we have
 $f_{V_{\le p-q-1}}(v_{p-q-1})\ge f_{V_{\ge p+q+1}}(v_{p+q+1})$,
 which implies
 $1=f_{V_0}(v_0) \ge f_{V_{\ge p+q+1}}(v_{p+q+1})=f_{V_{\ge l=1}}(v_{l+1})$. It is a
 contradiction. If $l$ is even, let $p=\frac{l}{2}$ and
 $q=\frac{l}{2}-1$.  Similarly, by   (ii) of  Corollary~\ref{cor3-3}, we have
$1=f_{V_0}(v_0) \ge f_{V_{\ge l=1}}(v_{l+1})$. It is a
 contradiction. So the assertion holds.

{\bf Case 3:} $d_2>d_k$. We have the following  Claim: $z_k>d_k-2$.
In fact, suppose that $z_k=d_k-2$. Then there exists a $3\le s\le k$
such that $z_{s-1}>z_s=\cdots =z_k=d_k-2$. If $k+s$ is odd, let
$p=\frac{k+s-1}{2}$ and $q=\frac{k-s+1}{2}$. Then
$f_{V_{p-i}}(v_{p-i})=f_{V_{p+i}}(v_{p+i}$ for $i=1, \cdots, q-1$
and
$$f_{V_{p-q}}(v_{p-q})=f_{V_{s-1}}(v_{s-1})=2^{z_{s-1}}>2^{z_k}
=f_{V_{k}}(v_{k})=f_{V_{p+q}}(v_{p+q}).$$
 Then by (i) of Corollary~\ref{cor3-3}, we have
 $$f_{V_{\le s-2}}(v_{s-2})=f_{V_{\le p-q-1}}(v_{p-q-q})>f_{V_{\ge
 p+q+1}}(v_{p+q+1})=f_{V_{k+1}}(v_{k+1})=1.$$
 It is a contradiction. If $k+s$ is even, by similar
  method and applying (ii) of Corollary~\ref{cor3-3}, we also get
  the contradiction.
 Hence the Claim holds.

 Further, there exists two integers $2\le t\le l\le k-1$  such that $
z_{t-1}>z_t=d_k-2$ and $z_t=\cdots= z_l<z_{l+1}$. Then
$f_{V_{t-1}}(v_{t-1})=2^{z_{t-1}}>2^{z_t}=f_{V_t}(v_t)$. Hence by
(i) of Corollary~\ref{cor3-3}, $f_{V_{\le t-2}}(v_{t-2})\le
f_{V_{\ge t+1}}(v_{t+1})$. Therefore, for any $1\le j\le t-2$ we
have
$$f_{V_{\le j-1}}(v_{j-1})< f_{V_{\le t-2}}(v_{t-2})\le
f_{V_{\ge t+1}}(v_{t+1})\le f_{V_{\le j+2}}(v_{j+2}).$$ Hence by (i)
of Corollary~\ref{cor3-3}, we have $f_{V_j}(v_{j})\ge
 f_{V_{j+1}}(v_{j+1})$, i.e.,  $2^{z_j}\ge 2^{z_{j+1}}$.
 So $z_1\ge z_2\ge \cdots\ge z_{t-1}>z_t$.

Since $z_l<z_{l+1}$,  we have $f_{V_{l}}(v_l)<f_{V_{l+1}}(v_{l+1})$.
By (i) of Corollary~\ref{cor3-3}, $f_{V_{\le
l-1}}(v_{l-1})>f_{V_{\ge l+2}}(v_{l+2})$. Then for any $l+1\le j\le
k-1$, we have
$$f_{V_{\le j-1}}(v_{j-1})\ge f_{V_{\le l-1}}(v_{l-1})\ge
f_{V_{\ge l+2}}(v_{l+2})> f_{V_{\le j+2}}(v_{j+2}).$$ Hence by (i)
of Corollary~\ref{cor3-3}, we have $f_{V_j}(v_{j})\le
 f_{V_{j+1}}(v_{j+1})$, i.e.,  $2^{z_j}\le 2^{z_{j+1}}$.
 So $z_l< z_{l+1}\le \cdots\le z_{k}$. We finish our proof.
\end{proof}

Similarly, there is a property of a maximum optional caterpillar
with a given degree sequence.

\begin{theorem}
Let $\pi=(d_1, d_2,\cdots, d_k,\cdots d_{n})$ be the degree sequence
of a tree with $d_1\ge d_2\ge \cdots \ge d_k\ge 2$ and
$d_{k+1}=\cdots =d_n=1$, where $k\ge 3$. If $C(z_1,z_2,\cdots,z_k)$
is a maximal optional caterpillar in $\mathcal{C}_{\pi}$, then there
exists an integer $1\le t\le k-1$ such that
\begin{eqnarray*}
z_1\le z_2\cdots \le z_{t-1}<z_t
\end{eqnarray*}
and
\begin{eqnarray*}
z_t\ge z_{t+1}\cdots \ge z_k.
\end{eqnarray*}
\end{theorem}

\begin{proof} The proof is similar to that of  Theorem~\ref{theorem3-5} and
omitted.
\end{proof}

\section{The optimal minimal  trees with many leaves}
In this section, for a given degree sequence $\pi=(d_1, d_2,\cdots,
d_k,\cdots d_{n})$ with at least $n-5$ leaves, we give the minimal
optimal trees with the minimum number
 of subtrees in ${\mathcal{T}}_{\pi}$. Moreover, the minimal optimal trees may be not unique.

\begin{theorem}
Let $\pi=(d_1, d_2,\cdots, d_k,\cdots, d_{n})$ be tree degree
sequence with $n-k$ leaves for $2\le k\le 4$. Then the minimal tree
in ${\mathcal{T}}_{\pi}$ is unique. In other words,

(i). If $k=2$, then $\varphi(T)=2^{n-2}+2^{d_1-1}+2^{d_2-1}+n-2$ for
any  $T\in {\mathcal{T}}_{\pi}$.

(ii). If $k=3$, then for any $T\in {\mathcal{T}}_{\pi}$,
$$\varphi(T)\ge \varphi(C(d_1-2,d_3-2,d_2-2))=n-3+2^{d_1-1}+2^{d_2-1}
+2^{d_3-2}+2^{d_1+d_3-3}+2^{d_3+d_2-3}+2^{n-3}.$$ with equality if
and if $T$ is the caterpillar $C(d_1-2,d_3-2,d_2-2)$.

(iii). If $k=4$, then
\begin{eqnarray*}
\varphi(T) &\ge &\varphi(C(d_1-2,d_4-2,d_3-2,d_2-2))\\
&=&n-4+2^{d_1-1}+2^{d_2-1}+2^{d_3-2}+2^{d_4-2}+2^{d_1+d_4-3}
+2^{d_3+d_4-4}+2^{d_3+d_2-3}\\
&&+2^{d_1+d_4+d_3-5}+2^{d_2+d_3+d_4-5}+2^{n-4}
\end{eqnarray*}
with equality if and only if  $T$ is the caterpillar
$C(d_1-2,d_4-2,d_3-2,d_2-2)$.
\end{theorem}
\begin{proof}
 (i).  $k=2$.  Then $T$ must be  $C(d_1-2, d_2-2)$ or $C(d_2-1, d_1-2)$, it is to see
 that
$$ \varphi(C(d_1-2, d_2-2))=\varphi(C(d_2-1,
d_1-2))=2^{n-2}+2^{d_1-1}+2^{d_2-1}+n-2.$$

  (ii). $k=3$. Then $T$ must be one of $C(d_1-2,d_2-2,d_3-2)$,$C(d_1-2,d_3-2,d_2-2)$ and
  $C(d_2-2,d_1-2,d_3-2)$. By Theorem~\ref{theorem3-6},
  $$\varphi(T)\ge \varphi(C(d_1-2,d_2-2,d_3-2))$$
  with equality if and only if $T$ is $C(d_1-2,d_3-2,d_2-2))$ and
  $$\varphi(C(d_1-2,d_3-2,d_2-2))=n-3+2^{d_1-1}+2^{d_2-1}
 +2^{d_3-2}+2^{d_1+d_3-3}+2^{d_3+d_2-3}+2^{n-3}.$$

(iii). $k=4$.  Let $T^*$ be any minimal optimal tree in the set of
all trees with $\pi=(d_1, d_2,d_3, d_4,1, \cdots,1)$ and $d_4\ge 2$.
By Theorem~\ref{theorem2-1} and \ref{theorem3-5}, $T^*$ must be
caterpillar and be one of  $C(d_1-2,d_2-2,d_4-2,d_3-2)$,
$C(d_1-2,d_3-2,d_4-2,d_2-2)$ and  $C(d_1-2,d_4-2,d_3-2,d_2-2)$.
 Further, by lemma~\ref{lemma3-2}, $T^*$ must
   be $C(d_1-2,d_4-2,d_3-2,d_2-2)$ and
$$\varphi(T^*)= n-4+2^{d_1-1}+2^{d_2-1}+2^{d_3-2}+2^{d_4-2}+2^{d_1+d_4-3}+2^{d_3+d_4-4}+2^{d_3+d_2-3}+2^{d_1+d_4+d_3-5}+2^{d_2+d_3+d_4-5}+2^{n-4}.$$
This completes the proof.
\end{proof}

\begin{theorem}\label{theorem4-2}
Let $\pi=(d_1, d_2,\cdots, d_5,1, \cdots,1)$ be tree degree sequence
with $n-5$ leaves.

(i). If $2^{d_1}> 2^{d_3-1}(1+2^{d_2-1})$ and $d_4\neq d_5$, then
there is exact one minimal optional tree
$C(d_1-2,d_5-2,d_4-2,d_3-2,d_2-2)$
 in $ {\mathcal{T}}_{\pi}$

(ii). If $2^{d_1}=2^{d_3-1}(1+2^{d_2-1})$ or $d_4=d_5$, then there
are  exact two minimal optional trees
$C(d_1-2,d_5-2,d_4-2,d_3-2,d_2-2)$ and
$C(d_1-2,d_4-2,d_5-2,d_3-2,d_2-2)$ in $ {\mathcal{T}}_{\pi}$.

(iii). If $2^{d_1}<2^{d_3-1}(1+2^{d_2-1})$ and $d_4\neq d_5$, there
is exact one minimal optional tree
$C(d_1-2,d_4-2,d_5-2,d_3-2,d_2-2)$ in ${\mathcal{T}}_{\pi}$.
\end{theorem}

\begin{proof} Let $T^*$ be any minimal optimal tree in
${\mathcal{T}}_{\pi}$. By Theorem~\ref{theorem2-1} and
\ref{theorem3-5}, $T^*$ must be  caterpillar and
$C(d_1-2,x_2,x_3,x_4,d_2-2)$, where $(x_2,x_3,x_4)$ is a permutation
of $(d_3-2,d_4-2,d_5-2)$.

Further, by Lemma~\ref{lemma3-2} , $T^*$  must be
$C(d_1-2,d_5-2,d_4-2,d_3-2,d_2-2)$ or
$C(d_1-2,d_4-2,d_5-2,d_3-2,d_2-2)$. Moreover,
\begin{eqnarray*}
&&\varphi(C(d_1-2,d_5-2,d_4-2,d_3-2,d_2-2))-\varphi(C(d_1-2,d_4-2,d_5-2,d_3-2,d_2-2))\\
&=&2^{d_5-2}+2^{d_1+d_5-3}+2^{d_4-2}+2^{d_3+d_4-4}+2^{d_2+d_3+d_5-3}\\
&&-[2^{d_4-2}+2^{d_1+d_4-3}+2^{d_5-2}+2^{d_3+d_5-4}+2^{d_2+d_3+d_4-3}]\\
&=&(2^{d_5-2}-2^{d_4-2})[2^{d_1}-2^{d_3-1}(1+2^{d_2-1})].
\end{eqnarray*}
If $2^{d_1}> 2^{d_3-1}(1+2^{d_2-1})$ and $d_4\neq d_5$, then
$$\varphi(C(d_1-2,d_5-2,d_4-2,d_3-2,d_2-2))-\varphi(C(d_1-2,d_4-2,d_5-2,d_3-2,d_2-2))<0.$$
Hence (i)holds.

If $2^{d_1}< 2^{d_3-1}(1+2^{d_2-1})$ and $d_4\neq d_5$, then
$$\varphi(C(d_1-2,d_5-2,d_4-2,d_3-2,d_2-2))-\varphi(C(d_1-2,d_4-2,d_5-2,d_3-2,d_2-2))>0.$$
Hence (iii)holds.

If $2^{d_1}= 2^{d_3-1}(1+2^{d_2-1})$ or $d_4=d_5$, then
 $$\varphi(C(d_1-2,d_5-2,d_4-2,d_3-2,d_2-2))=\varphi(C(d_1-2,d_4-2,d_5-2,d_3-2,d_2-2)).$$
 Hence (ii) holds.
\end{proof}

\noindent {\bf Remark} From Theorem~\ref{theorem4-2}, we can see
that the minimal optimal trees depend on the values of all
components of the tree degree sequences and not unique, while the
maximal optimal tree is unique for a given tree degree sequence. It
illustrates that it is difficult to character the minimal optimal
trees for a given degree sequence of a tree.

\end{document}